\documentclass[a4paper, 12pt]{amsart}
\usepackage{amssymb}
\usepackage{amsthm}
\usepackage{amsmath}
\usepackage{multirow}
\usepackage[all]{xy}
\usepackage[numeric]{amsrefs}
\usepackage{hyperref}
\usepackage{graphicx,color}
\usepackage{tikz}
\usetikzlibrary{intersections, calc, patterns}

\newcommand{\fun}{\mathbb{F}_1}
\renewcommand{\Re}{\mathop{\mathrm{Re}}}
\renewcommand{\Im}{\mathop{\mathrm{Im}}}
\newcommand{\ilim}[1]{\mathop{\varinjlim}\limits_{#1}}

\newcommand{\sts}[1]{\mathcal{O}_{#1}}

\newcommand{\CR}{\mathfrak{CRing}}

\newcommand{\Set}{\mathfrak{Set}}
\newcommand{\mr}{\mathfrak{MR}}
\newcommand{\Mo}{\mathfrak{M}_0}
\newcommand{\Sch}{\textbf{Sch}}
\newcommand{\MSch}{\textbf{MSch}}
\newcommand{\mcX}{\mathcal{X}}
\newcommand{\Zf}[1]{\underline{#1}_\mathbb{Z}}
\newcommand{\Mf}[1]{\underline{#1}_{\Mo}}
\newcommand{\Zs}[1]{#1_\mathbb{Z}}
\newcommand{\Ms}[1]{#1_{\Mo}}
\newcommand{\Hom}{\mathop{\mathrm{Hom}}\nolimits}
\newcommand{\abszeta}[1]{\zeta^{\mathrm{abs}}_{#1}}

\newcommand{\abschi}{\chi_{\mathrm{abs}}}

\newcommand{\Spec}{\mathop{\mathrm{Spec}}}
\newcommand{\MSpec}{\mathop{\mathrm{spec}}}
\newcommand{\rank}{\mathop{\mathrm{rank}}}
\DeclareRobustCommand{\reldim}{\qopname\relax o{rel.dim}}
\newcommand{\Arg}{\mathop{\mathrm{Arg}}}

\newtheorem{theorem}{Theorem}[section]
\newtheorem{prop}[theorem]{Proposition}
\newtheorem{lemma}[theorem]{Lemma}
\newtheorem{cor}[theorem]{Corollary}
\newtheorem{conj}[theorem]{Conjecture}

\theoremstyle{definition}
\newtheorem{definition}[theorem]{Definition}
\newtheorem{eg}[theorem]{Example}
\newtheorem{remark}[theorem]{Remark}

\title[Absolute Euler Product]{The absolute Euler product representation of the absolute zeta function for a torsion free Noetherian $\fun$-scheme}
\date{}
\author{Takuki Tomita}
\thanks{This research was supported in part by KAKENHI 18H05233. This research is also supported by the KLL 2020 Ph.D. Program Research Grant.}

\address{Department of Mathematics, Faculty of Science and Technology, Keio University, 3-14-1 Hiyoshi, Kouhoku-ku, Yokohama 223-8522, Japan}
\email{takuki@keio.jp}


\begin{document}
\maketitle

\begin{abstract}
The absolute zeta function for a scheme $X$ of finite type over $\mathbb{Z}$ satisfying a certain condition is defined as the limit as $p\to 1$ of the zeta function of $X\otimes\mathbb{F}_p$.
In 2016, after calculating absolute zeta functions for a few specific schemes, Kurokawa suggested that an absolute zeta function for a general scheme of finite type over $\mathbb{Z}$ should have an infinite product structure which he called the absolute Euler product.
In this article, formulating his suggestion using a torsion free Noetherian $\fun$-scheme defined by Connes and Consani, we give a proof of his suggestion. Moreover, we show that each factor of the absolute Euler product is derived from the counting function of the $\fun$-scheme.
\end{abstract}

\setcounter{tocdepth}{1}
\tableofcontents

%
%
%
%
%
\section{Introduction}
%
%
%
%
%

The purpose of this article is to give the absolute Euler product representation of the absolute zeta function associated with a torsion free Noetherian $\fun$-scheme.

The hypothetical concept of $\fun$ was introduced originally by Tits~\cite{tits1957analogues} and later by Deninger and Kurokawa as an approach to the Riemann Hypothesis. They suggested that if there were to exist a hypothetical base field $\fun$ of $\mathbb{Z}$ satisfying $\mathbb{Z}\otimes_{\fun}\mathbb{Z}\not\cong\mathbb{Z}$ and we could treat $\Spec\mathbb{Z}$ as a curve over $\fun$, then the Riemann Hypothesis would follow in a similar way as that of Deligne's proof of the Weil conjecture~\cite{manin1995lectures}. This idea has stimulated constructions of candidates of schemes over $\fun$ (\textit{$\fun$-schemes}) and studies on their properties.

While there is yet no definitive theory of $\fun$-schemes, currently there are various candidates of the definition of an $\fun$-scheme. Among various studies of $\fun$-schemes, Deitmar~\cite{deitmar2005} defined an $\fun$-scheme as a monoid scheme, gluing spectra of commutative monoids just like the construction of a scheme.
Connes and Consani~\cite{CC2010} extended this idea and defined an $\fun$-scheme as the functor which has both information of a monoid scheme and a usual scheme over $\mathbb{Z}$. In this article, we adopt this definition of an $\fun$-scheme.

In parallel with the studies of $\fun$-schemes, the zeta function associated with an $\fun$-scheme, which is called the \textit{absolute zeta function}, has been developed. We define the \textit{counting function} of $X$ to be a function $N_X(t)$ which satisfies $N_X(p^m)=\#X_p(\mathbb{F}_{p^m})$ for any prime number $p$ and $m\in\mathbb{N}$, where $X_p:=X\otimes\mathbb{F}_p$. For a scheme $X$ of finite type over $\mathbb{Z}$ (an \textit{arithmetic scheme}) whose counting function is a \textbf{polynomial}, Soul\'e~\cite{soule2004} defined the absolute zeta function for $X$ treated as an $\fun$-scheme as
\[\abszeta{X}(s)=\lim_{p\to 1}(p-1)^{N_X(1)}Z(X,p^{-s})\quad (s\in\mathbb{R}),\]
where $Z(X,p^{-s})$ is the zeta function of $X_p$ defined as
\[Z(X,p^{-s}):=\exp\left(\sum_{m=1}^\infty \frac{N_X(p^m)}{m}p^{-sm}\right).\]
It is known that $\abszeta{X}(s)$ in this case is a rational function and can be continued meromorphically to the whole complex plane $\mathbb{C}$. 
Connes and Consani~\cite{CC2010} extended this definition to the absolute zeta function $\zeta_{\mcX/\fun}(s)$ for a torsion free Noetherian $\fun$-scheme $\mcX$.

In \cite{kurokawa2016azf-eng}, Kurokawa calculated absolute zeta functions for some specific arithmetic schemes according to Soul\'e's definition~\cite{kurokawa2016azf-eng}*{Exercise 7.2} and suggested that the absolute zeta function of a general arithmetic scheme has an infinite product structure as follows, which he called the \textit{absolute Euler product}. Note that the absolute zeta function of a general arithmetic scheme has not been defined yet, but we assume its existence in the following.  In this article, we denote $\mathbb{N}:=\{1,2,\ldots\}$ and $\mathbb{N}_0:=\mathbb{N}\cup\{0\}$.

\begin{conj}[Kurokawa's suggestion (= Conjecture \ref{conj: Kurokawa's suggestion})~{\cite{kurokawa2016azf-eng}*{\S 7.3}}]
Let $X$ be an arithmetic scheme.
Then, there should exist an infinite product expansion of $\abszeta{X}(s)$ called the {\rm absolute Euler product} of the form
\begin{equation}\label{eq: AEP-conj-intro}
    \abszeta{X}(s)=\left(\frac1s\right)^{\abschi(X)}\prod_{n=1}^\infty \left(1-\left(\frac1s\right)^n\right)^{-\kappa(n,X)},
\end{equation}
where $\kappa(n,X)\in\mathbb{Z}$ for any $n\in\mathbb{N}$, $\abschi(X):=N_X(1)$ is the absolute Euler characteristic\footnote{The reason why this is called the Euler characteristic is because this coincides with $\sum (-1)^i b_i$ ($b_i$ is the $i^\text{th}$ Betti number), which is given by formally substituting $m=0$ in $N_X(p^m)=\#X(\mathbb{F}_{p^m})=\sum (-1)^i (\alpha_{i,1}^m+\cdots+\alpha_{i,b_i}^m)$ (the Lefschetz trace formula).}, and $N_X(t)$ is the counting function of $X$. Moreover, the infinite product (\ref{eq: AEP-conj-intro}) converges absolutely for $\Re(s)>\dim X$.
\end{conj}

In this article, we formulate Kurokawa's suggestion for a certain class of $\fun$-schemes and give its proof. Moreover, we explicitly give the integer $\kappa(n,X)$ using the points of a monoid scheme associated with an $\fun$-scheme. The following is the main theorem of this article.

\begin{theorem}[= Theorem \ref{result: AEP for Noetherian}]\label{thm: main-thm-intro}
Let $\mcX$ be a torsion free Noetherian $\fun$-scheme and let $\Ms{X}$ (resp. $\Zs{X}$) be the geometric realisation of $\mcX|_{\Mo}$ (resp. $\mcX|_{\CR}$). Note that the definitions of $\Mo$ and $\CR$ will be given in the beginning of \S 2 and $\Ms{X}$ (resp. $\Zs{X}$) is a monoid scheme (resp. a scheme over $\mathbb{Z}$). Then we have
\[\zeta_{\mcX/\fun}(s)=\left(\frac1s\right)^{N_\mcX(1)}\prod_{n=1}^\infty \left(1-\left(\frac1s\right)^n\right)^{-\kappa(n,\Ms{X})}.\]
Here, $N_\mcX(t)\in\mathbb{Z}[t]$ is the counting function of $\mcX$ and
\[\kappa(n,\Ms{X}):=\sum_{x\in\Ms{X}}\sum_{j=0}^{r(x)}(-1)^{r(x)-j}\binom{r(x)}{j}\kappa_j(n), \quad \kappa_j(n):=\frac1n\sum_{m|n}\mu\left(\frac nm\right)j^m,\]
where $r(x):=\rank\sts{\Ms{X},x}^\times$ for $x\in\Ms{X}$ and $\mu$ is the M\"obius function.
Moreover, if $\Zs{X}$ is of finite type over $\mathbb{Z}$, the region of absolute convergence of this absolute Euler product is $\{s\in\mathbb{C}\mid |s|>\reldim\Zs{X}/\mathbb{Z}\}$, where $\reldim\Zs{X}/\mathbb{Z}$ is the relative dimension of $\Zs{X}$ over $\Spec\mathbb{Z}$.
\end{theorem}

In fact, each factor of the absolute Euler product is derived from the counting function. 
We state in Corollary \ref{cor: AEP-linear Mobius transform} that $\kappa(n,\Ms{X})$ is given as the image of the counting function $N_\mcX(t)$ with respect to the homomorphism $M_n\colon\mathbb{Z}[t]\to\sum_{a\in\mathbb{Z}}\mathbb{Z}\kappa_a(n)$.

The content of this article is as follows.
In \S\ref{section: F1 and counting function}, we review the definition of an $\fun$-scheme defined by Connes and Consani and properties of a torsion free Noetherian $\fun$-scheme.
In \S\ref{section: AZF-AEP}, we review the definition and properties of the absolute zeta function for a torsion free Noetherian $\fun$-scheme. Then, we state the main theorem as we mentioned in Theorem \ref{thm: main-thm-intro}.
Then, we prove the main theorem in \S\ref{section: proof of AEP}, after showing some key lemmas.
Lastly, we give some examples of the absolute Euler products of the absolute zeta functions for some torsion free Noetherian $\fun$-schemes in \S\ref{section: examples}.

%
%
%
%
%
\section{$\fun$-schemes and their counting functions}\label{section: F1 and counting function}
%
%
%
%
%

In this section, we review the definition of $\fun$-schemes constructed by Connes and Consani~\cite{CC2010}, and show some properties of the counting function of a torsion free Noetherian $\fun$-scheme. In this article, we take a \textit{monoid} to be a commutative multiplicative monoid with the identity $1$ and the absorbing element $0$ which maps any element to 0 by multiplication, and we denote the category of sets, that of commutative rings, and that of monoids by $\Set$, $\CR$, and $\Mo$, respectively.

\subsection{$\fun$-schemes}\label{section: F1}

First of all, we review the definition of $\fun$-schemes defined by Connes and Consani~\cite{CC2010}.

\begin{definition}[{\cite{CC2010}*{\S 4.1}}]\label{def: MR}
Let $\beta$ be the functor $\mathbb{Z}[\cdot]\colon\Mo\to\CR$ and $\beta^\ast$ be the forgetful functor $U\colon\CR\to\Mo$.
We define the category $\mr:=\Mo\cup_{\beta,\beta^\ast}\CR$ as follows;
\begin{itemize}
\item objects: $\mathrm{Ob}(\mr):=\mathrm{Ob}(\Mo)\sqcup\mathrm{Ob}(\CR)$;
\item for any $M,N\in\Mo$ and $R,S\in\CR$,\\
\quad $\Hom_{\mr}(M,N):=\Hom_{\Mo}(M,N)$,\quad $\Hom_{\mr}(R,S):=\Hom_{\CR}(R,S)$,\\
\quad $\Hom_{\mr}(M,R):=\Hom_{\CR}(\beta(M),R)\cong\Hom_{\Mo}(M,\beta^\ast(R))$,\\
\quad $\Hom_{\mr}(R,M):=\emptyset$.
\end{itemize}
\end{definition}

An $\fun$-scheme is a functor $\mr\to\Set$ which combines information of an $\Mo$-scheme, a $\mathbb{Z}$-scheme and a natural transformation which binds them.
In short, an $\Mo$-scheme (resp. a $\mathbb{Z}$-scheme) is a functor $\Mo\to\Set$ (resp. $\CR\to\Set$) which admits an open covering by representable subfunctors.
For further details, see \cite{CC2010}*{\S 3}.

\begin{definition}[{\cite{CC2010}*{Definition 4.7}}]
A functor $\mcX\colon\mr\to\Set$ is an \textit{$\fun$-scheme} if it satisfies the following conditions;
\begin{itemize}
\item the restriction $\mcX|_{\Mo}$ is an $\Mo$-scheme;
\item the restriction $\mcX|_{\CR}$ is a $\mathbb{Z}$-scheme;
\item for any field $K$, the evaluation $e_K\colon(\mcX|_{\Mo}\circ\beta^\ast)(K)\to\mcX|_{\CR}(K)$ is bijective.
\end{itemize}
\end{definition}

\begin{remark}[{\cite{CC2010}*{Proposition 3.17}}]
For any $\Mo$-scheme $\mathcal{F}\colon\Mo\to\Set$, there exists the unique monoid scheme $\Ms{X}$ (up to isomorphism) such that 
\[\mathcal{F}=\Hom_{\MSch}(\MSpec(\cdot),\Ms{X})=:\Mf{X},\]
where $\MSch$ is the category of monoid schemes. We call $\Ms{X}$ the \textit{geometric realisation} of $\mathcal{F}$.
Similarly, for any $\mathbb{Z}$-scheme $\mathcal{G}\colon\CR\to\Set$, there exists the unique scheme $\Zs{X}$ over $\mathbb{Z}$  (up to isomorphism) such that $\mathcal{G}=\Hom_{\Sch}(\Spec(\cdot),\Zs{X})=:\Zf{X}$, where $\Sch$ is the category of schemes over $\mathbb{Z}$. We also call $\Zs{X}$ the \textit{geometric realisation} of $\mathcal{G}$.
\end{remark}

\subsection{Counting functions of torsion free Noetherian $\fun$-schemes}\label{section: counting function of F1}

Next, we show some properties of the counting function of a torsion free Noetherian $\fun$-scheme. At first, we review the definition of a torsion free Noetherian $\fun$-scheme and the important property that its counting function is a polynomial.

\begin{definition}[{\cite{CC2010}*{Definition 4.12}}]
An $\Mo$-scheme $\mathcal{F}$ is \textit{Noetherian} if there exist $d\in\mathbb{N}$ and Noetherian monoids $M_1,\ldots,M_d$ satisfying
\[\mathcal{F}=\bigcup_{i=1}^d\Hom_{\Mo}(M_i,\cdot)\]
by open subfunctors $\Hom_{\Mo}(M_i,\cdot)$.
Here, $M_i$ is Noetherian if and only if $M_i$ is finitely generated, or equivalently $\mathbb{Z}[M_i]$ is a Noetherian ring.
A \textit{Noetherian} $\mathbb{Z}$-scheme is also defined similarly.

An $\fun$-scheme $\mcX$ is \textit{Noetherian} if $\mcX|_{\Mo}$ is a Noetherian $\Mo$-scheme and $\mcX|_{\CR}$ is a Noetherian $\mathbb{Z}$-scheme.
\end{definition}

\begin{definition}[{\cite{CC2010}*{\S 4.4}}]
Let $X$ be the geometric realisation of $\underline{X}\colon\Mo\to\Set$. $X$ is \textit{torsion free} if the group $\sts{X,x}^\times$ is torsion free for any $x\in X$.
Moreover, we also call an $\fun$-scheme $\mcX$ to be \textit{torsion free} if the geometric realisation of $\mcX|_{\Mo}$ is torsion free.
\end{definition}

The following theorem is the important property of a torsion free Noetherian $\fun$-scheme to define its absolute zeta function defined by Soul\'e.

\begin{theorem}[{\cite{CC2010}*{Theorem 4.13}}]\label{thm: property of Noether F1}
Let $\mcX$ be a torsion free Noetherian $\fun$-scheme, and let $\Ms{X}$ (resp. $\Zs{X}$) be the geometric realisation of $\mcX|_{\Mo}$ (resp. $\mcX|_{\CR}$).
Also, let $C_n:=\langle s\mid s^n=1\rangle$ be the cyclic group of order $n$ and $\fun[C_n]:=C_n\cup\{0\}$ be a monoid.
\begin{enumerate}
\item There exists $N_\mcX(t)\in\mathbb{Z}[t]$ such that for any $n\in\mathbb{N}$ \[\#\Mf{X}(\fun[C_n])=N_\mcX(n+1).\]
\item For any $q=p^m$, $\#\Zf{X}(\mathbb{F}_q)=N_\mcX(q)$.
\end{enumerate}
We call this polynomial $N_\mcX(t)$ the {\rm counting function} of $\mcX$.
\end{theorem}

\begin{remark}
$\fun[C_n]$ in Theorem \ref{thm: property of Noether F1} (1) is called the $n^{\text{th}}$ extension $\mathbb{F}_{1^n}$ of $\fun$. In parallel with (2), (1) means that $\#\Mf{X}(\mathbb{F}_{1^n})=N_\mcX(n+1)$.
\end{remark}

\begin{remark}[{\cite{CC2010}*{Theorem 4.13}}]\label{rem: counting function of Noet F1}
It is easy to check that the counting function $N_\mcX(t)$ in Theorem \ref{thm: property of Noether F1} is given by
\[N_\mcX(t):=\sum_{x\in\Ms{X}} (t-1)^{r(x)}=\sum_{x\in\Ms{X}}\sum_{j=0}^{r(x)}(-1)^{r(x)-j}\binom{r(x)}{j}t^j,\]
where $r(x):=\rank\sts{\Ms{X},x}^\times$ for any $x\in\Ms{X}$.
\end{remark}

\begin{cor}
Let $\mcX$ be a torsion free Noetherian $\fun$-scheme, and let $\Ms{X}$ be the geometric realisation of $\mcX|_{\Mo}$.
Then we have $\#\Ms{X}<\infty$.
\end{cor}

\begin{proof}
By Remark \ref{rem: counting function of Noet F1}, we have $N_\mcX(2)=\#\Ms{X}<\infty$.
\end{proof}

The counting function of a torsion free Noetherian $\fun$-scheme has the following property. We use this property later when we display the region of absolute convergence of the absolute Euler product using the relative dimension of the scheme associated with the $\fun$-scheme.
\begin{theorem}\label{thm: counting-dim}
Let $\mcX$ be a torsion free Noetherian $\fun$-scheme, and $N_\mcX(t)$ be its counting function.
Let $\Zs{X}$ be the geometric realisation of $\mcX|_{\CR}$. Assume that $\Zs{X}$ is of finite type over $\mathbb{Z}$.
Then, $N_\mcX(t)\in\mathbb{Z}[t]$ and
\[\deg N_\mcX = \reldim\Zs{X}/\mathbb{Z}.\]
\end{theorem}

To prove this theorem, we use the following lemma proved by Poonen.

\begin{lemma}[{\cite{poonen2017rational}*{Theorem 7.7.1 (i)}}]\label{lem: order of rational points}
Let $\pi\colon X\to Y$ be a morphism between schemes of finite type over $\mathbb{Z}$. Let $q$ be a prime power (i.e. $q=p^m$ for a prime number $p$), let $y\in Y(\mathbb{F}_q)$, let $X_y:=X\times_Y\Spec\kappa(y)$, and let $d=\dim X_y$.
Then, we have \[\#X_y(\mathbb{F}_q)=O(q^d).\]
\end{lemma}

\begin{proof}[Proof of Theorem \ref{thm: counting-dim}]
The fact that $N_\mcX(t)\in\mathbb{Z}[t]$ follows immediately from Theorem \ref{thm: property of Noether F1}.

Let $d:=\reldim\Zs{X}/\mathbb{Z}$. Fix a prime number $p$.
Let $Y=\Spec\mathbb{Z}$ and take $y\in Y(\mathbb{F}_q)$.
Note that $\Zs{X}$ is of finite type over $\mathbb{Z}$.
Put $X_y=\Zs{X}\times_Y\Spec\kappa(y)$, then we have $\dim X_y=d$.
By Lemma \ref{lem: order of rational points}, we have \[N_\mcX(q)=\#X_y(\mathbb{F}_q)=O(q^d),\] since $\#X_y(\mathbb{F}_q)=N_\mcX(q)$ by Theorem \ref{thm: property of Noether F1}. Since $N_\mcX(t)$ is a polynomial, we have $\deg N_\mcX=d=\reldim\Zs{X}/\mathbb{Z}$.
\end{proof}

\begin{cor}\label{cor: CF of torsion-free Noet F1}
Let $\mcX$ be a torsion free Noetherian $\fun$-scheme, and let $\Ms{X}$ (resp. $\Zs{X}$) be the geometric realisation of $\mcX|_{\Mo}$ (resp. $\mcX|_{\CR}$). Assume that $\Zs{X}$ is of finite type over $\mathbb{Z}$. Then the counting function of $\mcX$ is given explicitly as
\[N_\mcX(t)=\sum_{j=0}^{r}\left(\sum_{d=j}^{r}(-1)^{d-j}\binom dj \#I_d\right) t^j,\]
where $r:=\reldim X_{\mathbb{Z}}/\mathbb{Z}$, $r(x):=\rank\sts{\Ms{X},x}^\times$ and $I_d:=\{x\in \Ms{X}\mid r(x)=d\}$.
\end{cor}

\begin{proof}
By Theorem \ref{thm: counting-dim}, we have $\max\limits_{x\in\Ms{X}} r(x)=\deg N_\mcX=\reldim\Zs{X}/\mathbb{Z}=r$. Since $I_d=\emptyset$ for $d>r$, we have
\begin{align*}
N_\mcX(t)&=\sum_{x\in\Ms{X}}(t-1)^{r(x)}=\sum_{d=0}^{r}\#I_d(t-1)^d\\
&=\sum_{d=0}^{r}\#I_d\sum_{j=0}^d (-1)^{d-j}\binom dj t^j
=\sum_{j=0}^{r}\left(\sum_{d=j}^{r}(-1)^{d-j}\binom dj \#I_d\right) t^j
\end{align*}
by Remark \ref{rem: counting function of Noet F1}.
\end{proof}

%
%
%
%
%
\section{Absolute zeta function and the main theorem}\label{section: AZF-AEP}
%
%
%
%
%

In this section, we state the main theorem of this article. At first, we review the definition and property of the absolute zeta function for a torsion free Noetherian $\fun$-scheme which was introduced by Connes and Consani~\cite{CC2010}. Then, after reviewing Kurokawa's suggestion~\cite{kurokawa2016azf-eng}*{\S 7}, we state our main theorem (Theorem \ref{result: AEP for Noetherian}) that the absolute zeta function for a torsion free Noetherian $\fun$-scheme has the absolute Euler product representation.
Moreover, in Corollary \ref{cor: AEP-linear Mobius transform}, we explicitly determine each factor of the absolute Euler product in terms of the counting function of the $\fun$-scheme.

\subsection{The absolute zeta function of a torsion free Noetherian $\fun$-scheme}\label{section: def of AZF}

In \cite{soule2004}, Soul\'e defined the absolute zeta function for an arithmetic scheme whose counting function is a polynomial as the limit of $Z(X,p^{-s})$ as $p\to 1$ as explained in \S 1. Note that absolute zeta functions for general arithmetic schemes cannot be defined similarly. 

Using this definition, Connes and Consani~\cite{CC2010} defined the absolute zeta function $\zeta_{\mcX/\fun}(s)$ for a torsion free Noetherian $\fun$-scheme $\mcX$ as the absolute zeta function for the geometric realisation of $\mcX|_{\CR}$.

\begin{definition}[{\cite{CC2010}*{\S 2}}]
Let $\mcX$ be a torsion free Noetherian $\fun$-scheme and $\Zs{X}$ be the geometric realisation of $\mcX|_{\CR}$.
Note that there exists the polynomial counting function $N_\mcX(t)$ of $\mcX$ by Theorem \ref{thm: property of Noether F1}. We define the function
\[\zeta_{\mcX/\fun}(s):=\abszeta{\Zs{X}}(s)=\lim_{p\to 1}  (p-1)^{N_\mcX(1)}\exp\left(\sum_{m=1}^\infty \frac{N_\mcX(p^m)}{m}p^{-sm}\right)\]
and call it the \textit{absolute zeta function} for $\mcX$.
\end{definition}

By this definition, we immediately obtain the following proposition.

\begin{prop}[{\cite{soule2004}*{Lemme 1}}]\label{prop: AZF of Noet F1-scheme}
Let $\mcX$ be a torsion free Noetherian $\fun$-scheme and $N_\mcX(t)$ be the counting function of $\mcX$. Put $N_\mcX(t)=\sum\limits_{k=0}^r a_k t^k$ (cf. Theorem \ref{thm: property of Noether F1}). Then, the absolute zeta function of $\mcX$ is expressed as \[\zeta_{\mcX/\fun}(s)=\prod_{k=0}^r(s-k)^{-a_k},\]
and can be continued meromorphically to the whole complex plane $\mathbb{C}$.
\end{prop}

In \cite{CC2010}, Connes and Consani gave a finite product representation of the absolute zeta function of a torsion free Noetherian $\fun$-scheme $\mcX$ which runs over the points of the geometric realisation of $\mcX|_{\Mo}$. Each factor of the product is expressed using the Kurokawa tensor product as defined below.

\begin{definition}[\cite{kurokawa1992multiple}, \cite{manin1995lectures}]\label{def: kurokawa-tensor}
For $i\in\{1,\ldots,r\}$, let $\Phi_i$ be a finite subset of $\mathbb{C}$ and $m_i\colon\Phi_i\to\mathbb{Z}$ be a function.
Let
\[Z_i(s):=\prod_{\rho\in\Phi_i}(s-\rho)^{m_i(\rho)}.\]
We define the \textit{Kurokawa tensor product} as follows;
\[Z_1(s)\otimes\cdots\otimes Z_r(s):=\prod_{(\rho_1,\ldots,\rho_r)\in\Phi_1\times\cdots\times\Phi_r}(s-(\rho_1+\cdots+\rho_r))^{m(\rho_1,\ldots,\rho_r)},\]
where \[m(\rho_1,\ldots,\rho_r):=\begin{cases}m_1(\rho_1)\cdots m_r(\rho_r) & (\text{$\Im(\rho_i)\geq 0$ for each $i$})\\ (-1)^{r-1}m_1(\rho_1)\cdots m_r(\rho_r) & (\text{$\Im(\rho_i)< 0$ for each $i$})\\ 0 & (\text{otherwise}).\end{cases}\]
\end{definition}

\begin{theorem}[{\cite{CC2010}*{Theorem 4.13}}]\label{thm: property of AZF for Noether}
Let $\mcX$ be a torsion free Noetherian $\fun$-scheme. Then we have
\[\zeta_{\mcX/\fun}(s)=\prod_{x\in\Ms{X}}\frac1{\left(1-\frac1s\right)^{\otimes r(x)}},\]
where $r(x):=\rank \sts{\Ms{X},x}^\times$ and $\otimes$ is the Kurokawa tensor product.
\end{theorem}

\subsection{Absolute Euler product}

Now, we explain the infinite product structure of an absolute zeta function called the absolute Euler product, which was introduced by Kurokawa~\cite{kurokawa2016azf-eng}*{\S 7}.

An Euler product is an infinite product which runs over prime numbers. It is well-known that the Riemann zeta function has the following Euler product representation;
\[\zeta(s)=\prod_{p\mathrm{:prime}} (1-p^{-s})^{-1}.\]
As a generalisation of the Euler product, for a ring $R$ finitely generated over $\mathbb{Z}$, the zeta function $\zeta_R(s)$ of $R$ (the Hasse zeta function for $\Spec R$) has the following Euler product representation;
\[\zeta_R(s):=\prod_{\mathfrak{a}\in\text{m-Spec}R}(1-(\# R/\mathfrak{a})^{-s})^{-1}=\prod_{p\mathrm{:prime}}\prod_{n=1}^\infty (1-p^{-ns})^{-\kappa(p,n;R)},\]
where $\kappa(p,n;R):=\#\{\mathfrak{a}\in\text{m-Spec}R\mid \# R/\mathfrak{a} = p^n\}$.

\begin{eg}
When $R=\mathbb{Z}[T]$, we have
\[\kappa(p,n;\mathbb{Z}[T])=\frac1n\sum_{m|n}\mu\left(\frac nm\right)p^m\]
and $\zeta_{\mathbb{Z}[T]}(s)=\zeta(s-1)$.
Here, the map $\mu\colon\mathbb{N}\to\{-1,0,1\}$ is called the \textit{M\"obius function} and is defined as
\[\mu(n):=\begin{cases}0 & (\text{$n$ has a squared prime factor})\\ (-1)^k & (\text{$n$ has $k$ distinct prime factors})\end{cases}\]
for any $n\in\mathbb{N}$.
\end{eg}

Next, we consider the Euler product of $Z(X,p^{-s})$. Let $X$ be an arithmetic scheme. Then, we have
\[Z(X,p^{-s})=\prod_{x\in|X|}(1-\#k(x)^{-s})^{-1}=\prod_{n=1}^\infty (1-p^{-ns})^{-\kappa(p,n;X)},\]
where $\kappa(p,n;X):=\#\{x\in |X|\mid \#k(x)=p^n\}$ and $k(x)$ is the residue field of $x$. Thus, the Euler product of $Z(X,p^{-s})$ is
\begin{equation}\label{eq: EP of cong-zeta}
Z(X,p^{-s})=\prod_{l\mathrm{:prime}} \prod_{n=1}^\infty (1-l^{-ns})^{-\kappa(l,n;X)},
\end{equation}
where
\[\kappa(l,n;X)=\begin{cases}\kappa(p,n;X) & (l=p)\\0&(l\ne p).\end{cases}\]

If $p\to 1$ in this representation, it seems that the left-hand side of (\ref{eq: EP of cong-zeta}) becomes the absolute zeta function of $X$ and the right-hand side of (\ref{eq: EP of cong-zeta}) becomes a certain infinite product. However, Kurokawa pointed out that it does not actually work well especially when $X=\mathbb{G}_m$ or $\mathbb{G}_m^2$ since $1-p^{-ns}\to 0$ and $\kappa(p,n;X)\to 0$ as $p\to 1$~\cite{kurokawa2016azf-eng}*{p.131}. Then, through calculations of some specific cases, Kurokawa proposed that the Euler product of the absolute zeta function should be given as follows.

\begin{conj}[Kurokawa's suggestion~{\cite{kurokawa2016azf-eng}*{\S 7.3}}]\label{conj: Kurokawa's suggestion}
Let $X$ be an arithmetic scheme.
Then, there should exist an infinite product expansion of $\abszeta{X}(s)$ called the {\rm absolute Euler product} of the form
\begin{equation}\label{eq: AEP-conj}
    \abszeta{X}(s)=\left(\frac1s\right)^{\abschi(X)}\prod_{n=1}^\infty \left(1-\left(\frac1s\right)^n\right)^{-\kappa(n,X)},
\end{equation}
where $\kappa(n,X)\in\mathbb{Z}$ for any $n\in\mathbb{N}$, $\abschi(X):=N_X(1)$ is the absolute Euler characteristic, and $N_X(t)$ is the counting function of $X$. Moreover, the infinite product (\ref{eq: AEP-conj}) converges absolutely for $\Re(s)>\dim X$.
\end{conj}

\begin{remark}
Since the absolute zeta function for a general arithmetic scheme is not currently defined, this suggestion assumes that there exists the absolute zeta function for any general arithmetic scheme.
\end{remark}

We now give the main theorem of this article. It is a formulation of Kurokawa's suggestion using torsion free Noetherian $\fun$-schemes. Moreover, we give an explicit form of $\kappa(n,X)$ using the points of the monoid scheme associated with the $\fun$-scheme.

\begin{theorem}\label{result: AEP for Noetherian}
Let $\mcX$ be a torsion free Noetherian $\fun$-scheme and let $\Ms{X}$ (resp. $\Zs{X}$) be the geometric realisation of $\mcX|_{\Mo}$ (resp. $\mcX|_{\CR}$). Then we have
\[\zeta_{\mcX/\fun}(s)=\left(\frac1s\right)^{N_\mcX(1)}\prod_{n=1}^\infty \left(1-\left(\frac1s\right)^n\right)^{-\kappa(n,\Ms{X})}.\]
Here, $N_\mcX(t)\in\mathbb{Z}[t]$ is the counting function of $\mcX$ and
\[\kappa(n,\Ms{X}):=\sum_{x\in\Ms{X}}\sum_{j=0}^{r(x)}(-1)^{r(x)-j}\binom{r(x)}{j}\kappa_j(n), \quad \kappa_j(n):=\frac1n\sum_{m|n}\mu\left(\frac nm\right)j^m,\]
where $r(x):=\rank\sts{\Ms{X},x}^\times$.
Moreover, if $\Zs{X}$ is of finite type over $\mathbb{Z}$, the region of absolute convergence of this absolute Euler product is $\{s\in\mathbb{C}\mid |s|>\reldim\Zs{X}/\mathbb{Z}\}$.
\end{theorem}

We give the proof of Theorem \ref{result: AEP for Noetherian} in \S \ref{subsection: proof of main thm}.
In fact, $\kappa(n,\Ms{X})$ is given as the image of the counting function of $\mcX$ with respect to the following homomorphism.
\begin{definition}
For any $n\in\mathbb{N}$, we define the homomorphism of $\mathbb{Z}$-modules
\[M_n\colon \mathbb{Z}[t]\to \sum_{a\in\mathbb{Z}}\mathbb{Z}\kappa_a(n)\]
such that $M_n(t^a):=\kappa_a(n)$.
\end{definition}

Using this homomorphism $M_n$, Theorem \ref{result: AEP for Noetherian} can be represented only by the counting function $N(t)$.
\begin{cor}\label{cor: AEP-linear Mobius transform}
Let $\mcX$ be a torsion free Noetherian $\fun$-scheme and let $N_\mcX(t)$ be the counting function of $\mcX$. Then we have
\[\zeta_{\mcX/\fun}(s)=\left(\frac1s\right)^{N_\mcX(1)}\prod_{n=1}^\infty \left(1-\left(\frac1s\right)^n\right)^{-M_n(N_\mcX(t))}.\]
Moreover, the region of absolute convergence of this absolute Euler product is $\{s\in\mathbb{C}\mid |s|>\deg N_\mcX\}$.
\end{cor}

\begin{proof}
By Remark \ref{rem: counting function of Noet F1}, we have $\kappa(n,\Ms{X})=M_n(N_\mcX(t))$. Thus, the absolute Euler product follows from Theorem \ref{result: AEP for Noetherian}. The region of absolute convergence can be shown in the same way as that of Theorem \ref{result: AEP for Noetherian}.
\end{proof}

%
%
%
%
%
\section{Proof of the main theorem}\label{section: proof of AEP}
%
%
%
%
%

In this section, we give a proof of the main theorem (Theorem \ref{result: AEP for Noetherian}).

\subsection{Lemmas}

To prove Theorem \ref{result: AEP for Noetherian}, we first prove the core formula of the absolute Euler product representation and its region of absolute convergence.

\begin{lemma}[{\cite{kurokawa2016azf-eng}*{Exercise 7.1}}]\label{lem: core of AEP-formula}
Let $a\in\mathbb{Z}$, $n\in\mathbb{N}$, and
\[\kappa_a(n):=\frac1n\sum_{m|n}\mu\left(\frac nm\right)a^m.\]
Then, we have $\kappa_a(n)\in\mathbb{Z}$ for any $n\in\mathbb{N}$ and $a\in\mathbb{Z}$, and it holds that in $\mathbb{Z}[\![u]\!]$
\[\prod_{n=1}^\infty (1-u^n)^{\kappa_a(n)}=1-au.\]
\end{lemma}

\begin{remark}
This formula appears in the definition of the Euler transform~\cite{encyc-integer-sequences}.
\end{remark}

\begin{proof}[Proof of Lemma \ref{lem: core of AEP-formula}]
At first, we show $\kappa_a(n)\in\mathbb{Z}$. We prove this in a different way from Kurokawa's proof~{\cite{kurokawa2016azf-eng}*{Exercise 7.1}}, using the following property of the unit group of $\mathbb{Z}/p^{e+1}\mathbb{Z}$;
\[
\left(\mathbb{Z}/p^{e+1}\mathbb{Z}\right)^\times\cong
\begin{cases}
\mathbb{Z}/(p-1)\mathbb{Z} \times \mathbb{Z}/p^e\mathbb{Z} & (p\text{ : odd})\\
\{1\} & (p=2, e=0)\\
\mathbb{Z}/2\mathbb{Z} \times \mathbb{Z}/2^{e-1}\mathbb{Z} & (p=2, e\geq 1)
\end{cases}
\]
for any $e\in\mathbb{N}_0$ and prime number $p$~\cite{serre2012course}*{Chapter II, Theorem 2}.
Since $y^{p^{e+1}}-y^{p^e}=y^{p^e}(y^{p^e(p-1)}-1)$, we have
\begin{equation}
    y^{p^{e+1}}\equiv y^{p^e}\pmod{p^{e+1}}
    \label{eq: mod p^e}
\end{equation}
for any $y\in\mathbb{Z}$ and $e\in\mathbb{N}_0$.

Since $\mathbb{Z}=\bigcap_{p\mathrm{:prime}} \mathbb{Z}_{(p)}$,
it suffices to show that $\kappa_a(n)\in\mathbb{Z}_{(p)}$ for any prime number $p$. Fix any $a$ and $p$.
If $p\nmid n$, $\kappa_a(n)\in\mathbb{Z}\left[\frac1n\right]\subset\mathbb{Z}_{(p)}$ holds by definition.
We assume that $p\mid n$. By the equation (\ref{eq: mod p^e}), putting $n=p^\nu u$ ($\nu=v_p(n)$, $u\in\mathbb{Z}$),
\begin{align*}
\kappa_a(n)&=\frac1n\sum_{m\mid n}\mu\left(\frac nm\right)a^m
=\frac1{p^\nu u}\sum_{m\mid p^\nu u}\mu\left(\frac {p^\nu u}m\right)a^m\\
&\overset{(\ast)}{=}\sum_{d\mid u}\frac1{p^\nu u}\mu\left(\frac ud\right)\left(a^{p^\nu d}-a^{p^{\nu-1}d}\right)\\
&\in\mathbb{Z}_{(p)} \qquad \left(\because\ a^{p^\nu d}-a^{p^{\nu-1}d}\equiv 0\pmod{p^\nu\mathbb{Z}_{(p)}}\right).
\end{align*}
Here, $(\ast)$ follows by dividing into two cases; $m=p^{\nu-1}d$ and $m=p^\nu d$, where $\frac{p^\nu u}{m}$ does not have any squared prime factor.

Next, we formally calculate the infinite product representation of $1-au$;
\begin{align*}
\log\left(\prod_{n=1}^\infty (1-u^n)^{\kappa_a(n)}\right)
&= \sum_{n=1}^\infty \kappa_a(n)\log(1-u^n)
= -\sum_{n=1}^\infty \sum_{k=1}^\infty \frac{n\kappa_a(n)}{nk}u^{nk}\\
&\overset{\mathrm{(a)}}{=} -\sum_{m=1}^\infty \frac{1}{m}\left(\sum_{n|m} n\kappa_a(n)\right)u^m\\
&= -\sum_{m=1}^\infty \frac{1}{m}\left(\sum_{n|m} \sum_{l|n}\mu\left(\frac nl\right)a^l\right)u^m\\
&\overset{\mathrm{(b)}}{=}-\sum_{m=1}^\infty \frac{1}{m}a^m u^m
= \log(1-au).
\end{align*}
Here, we put $m=nk$ in (a) and we use the M\"obius inversion formula in (b).
\end{proof}

\begin{lemma}\label{lem: core of AEP-convergence}
Let $a\in\mathbb{Z}$ and $u\in\mathbb{C}$.
If $a\geq 2$, then the region of absolute convergence of the infinite product in Lemma \ref{lem: core of AEP-formula}
\[\prod_{n=1}^\infty (1-u^n)^{\kappa_a(n)}\]
is $\{u\in\mathbb{C}\mid |u|<\frac1a\}$. If $a=1$, then its region of absolute convergence is $\mathbb{C}$.
\end{lemma}

To prove this lemma, we show the following property of $\kappa_a(n)$.

\begin{prop}\label{prop: order of kappa}
For any $a\in\mathbb{N}_0$ and $n\in\mathbb{N}$,
\[\left|\kappa_a(n)-\frac{a^n}n\right|\leq\frac{a^{\left\lfloor n/2\right\rfloor+1}}n.\]
\end{prop}

\begin{proof}
At first, we consider the easy cases. We have $\kappa_0(n)=0$ and $\kappa_a(1)=a$. When $a=1$, we have
\[\kappa_1(n)=\frac1n\sum_{m|n}\mu\left(\frac nm\right)=\frac1n\delta_{1n}=\begin{cases}1 & (n=1)\\ 0 & (n\ne 1)\end{cases}\]
by the property of the M\"obius function.

Let $a, n\geq 2$. Then we have
\[|n\kappa_a(n)-a^n|\leq\sum_{\substack{m|n\\m\ne n}}a^m \leq \sum_{m=0}^{\left\lfloor n/2\right\rfloor}a^m=\frac{a^{\left\lfloor n/2\right\rfloor+1}-1}{a-1}\leq a^{\left\lfloor n/2\right\rfloor+1}.\]
Thus, the claim follows.
\end{proof}

\begin{cor}
The values and signs of $\kappa_a(n)$'s are as in Table \ref{table: kappa}.
\begin{table}[ht]
    \centering
    \begin{tabular}{c||cc|ccc|cc}
        $n$\textbackslash $a$ & \multicolumn{2}{c|}{$\cdots$ \quad $-2$} & $-1$ & $0$ & $1$ & $2$ & $\cdots$\\\hline\hline
        $1$ & \multicolumn{2}{c|}{\ $\cdots$\quad\ $-2$\ \ } & $-1$ & $0$ & $1$ & $2$ & $\cdots$\\\hline
        $2$ & \multicolumn{2}{c|}{\multirow{3}{*}{\begin{tabular}{cl}$+$ & {\rm (if $2|n$)}\\$-$ & {\rm (if $2\nmid n$)}\end{tabular}}} & $1$ & $0$ & $0$ & \multicolumn{2}{c}{\multirow{3}{*}{$+$}}\\
        $3$ & & & $0$ & $0$ & $0$ & &\\
        \vdots & & & \vdots & \vdots & \vdots & &\\
    \end{tabular}
    \caption{the values and signs of $\kappa_a(n)$'s}
    \label{table: kappa}
\end{table}
\end{cor}

\begin{proof}
By Proposition \ref{prop: order of kappa}, it suffices to show the case when $a\leq -1$.

Let $a=-1$. We have $\kappa_{-1}(1)=-1$ and $\kappa_{-1}(2)=1$, hence we consider the case when $n\geq 3$. Then, we have
\[\kappa_{-1}(n)=\frac1n\sum_{m|n}\mu\left(\frac nm\right)(-1)^m = 0.\]
Indeed, since the number of divisors of $n$ is at most $\left\lfloor\frac n2\right\rfloor+1$, $\kappa_{-1}(n)$ is in $-\frac12-\frac1{n}\leq\kappa_{-1}(n)\leq\frac12+\frac1{n}$, then we have $\kappa_{-1}(n)=0$ since $\kappa_{-1}(n)\in\mathbb{Z}$ by Lemma \ref{lem: core of AEP-formula}.

Let $a\geq 2$. Since
\[n\kappa_{-a}(n)=\sum_{m|n}\mu\left(\frac nm\right)(-1)^m a^m = (-1)^n a^n+\cdots-\mu(n)a,\]
we have $\kappa_{-a}(n)>0$ when $2\mid n$ and $\kappa_{-a}(n)<0$ when $2\nmid n$ in the similar way to prove Proposition \ref{prop: order of kappa}.
\end{proof}

\begin{proof}[Proof of Lemma \ref{lem: core of AEP-convergence}]
When $a=1$, $\kappa_1(n)=\delta_{1n}$ and thus the infinite product converges for any $u\in\mathbb{C}$. We assume $a\geq 2$ in the following. At first, we show that the infinite product converges absolutely at least for $|u|<\frac1a$. Since
\[\prod_{n=1}^\infty (1-u^n)^{\kappa_a(n)}=\prod_{n=1}^\infty \left(1+((1-u^n)^{\kappa_a(n)}-1)\right),\]
it suffices to show that
\[\sum_{n=1}^\infty |(1-u^n)^{\kappa_a(n)}-1|\]
converges for $|u|<\frac1a$. At first, we calculate the upper bound of the binomial coefficient $\binom{\kappa_a(n)}{m}$.
\[\binom{\kappa_a(n)}{m}\leq\frac{\kappa_a(n)^m}{m!}
<\frac{\kappa_a(n)^m}{\sqrt{2\pi m}\left(\frac me\right)^m}
=\frac{\kappa_a(n)^m e^m}{\sqrt{2\pi}m^{m+\frac12}}
\leq \frac1{\sqrt{2\pi}}\left(\frac{e\kappa_a(n)}{m}\right)^m.\]
Since $\displaystyle{\kappa_a(n)\leq\frac{a^n+a^{\lfloor n/2\rfloor+1}}n\leq \frac{2a^n}n}$ by Proposition \ref{prop: order of kappa}, we have
\[\binom{\kappa_a(n)}{m}<\frac1{\sqrt{2\pi}}\left(\frac{e\kappa_a(n)}{m}\right)^m\leq \frac1{\sqrt{2\pi}}\left(\frac{2ea^n}{nm}\right)^m.\]
Putting $|u|=\frac1a-\varepsilon$ $\left(0<\varepsilon<\frac1a\right)$, we have
\begin{align*}
\sum_{n=1}^\infty |(1-u^n)^{\kappa_a(n)}-1|
&= \sum_{n=1}^\infty \left|\sum_{m=1}^{\kappa_a(n)}(-1)^m\binom{\kappa_a(n)}{m}u^{nm}\right|\\
&\leq \sum_{n=1}^\infty \sum_{m=1}^{\kappa_a(n)}\binom{\kappa_a(n)}{m}|u|^{nm}
= \sum_{n=1}^\infty \sum_{m=1}^{\kappa_a(n)}\binom{\kappa_a(n)}{m}\left(\frac1a-\varepsilon\right)^{nm}\\
&< \sum_{n=1}^\infty \sum_{m=1}^{\kappa_a(n)}\frac1{\sqrt{2\pi}}\left(\frac{2ea^n}{nm}\right)^m \left(\frac{1-\varepsilon a}{a}\right)^{nm}\\
&\leq \sum_{n=1}^\infty \sum_{m=1}^{\kappa_a(n)}\frac1{\sqrt{2\pi}}\left(\frac{2e(1-\varepsilon a)^n}{n}\right)^m.
\end{align*}
Here, we put $\displaystyle{c_\varepsilon=\frac1{1-\varepsilon a}}$ ($c_\varepsilon>1$) and $\displaystyle{r_n=\frac{2e(1-\varepsilon a)^n}{n}=\frac{2e}{nc_\varepsilon^n}}$. Then,
\[\sum_{n=1}^\infty |(1-u^n)^{\kappa_a(n)}-1|
< \sum_{n=1}^\infty \frac1{\sqrt{2\pi}}\sum_{m=1}^{\kappa_a(n)}r_n^m
= \frac1{\sqrt{2\pi}}\sum_{n=1}^\infty \frac{r_n^{\kappa_a(n)}-r_n}{r_n-1}.\]
We take and fix sufficiently large $N\in\mathbb{N}$ satisfying $c_\varepsilon^n>n$ and $r_n<1$ for any $n\geq N$.
\begin{align*}
\sum_{n=1}^\infty |(1-u^n)^{\kappa_a(n)}-1|
&< \frac1{\sqrt{2\pi}}\left(\sum_{n=1}^{N-1} \frac{r_n^{\kappa_a(n)}-r_n}{r_n-1}+\sum_{n=N}^\infty \frac{r_n-r_n^{\kappa_a(n)}}{1-r_n}\right)\\
&\leq \frac1{\sqrt{2\pi}}\left(\sum_{n=1}^{N-1} \frac{r_n^{\kappa_a(n)}-r_n}{r_n-1}+\sum_{n=N}^\infty \frac{r_n}{1-r_N}\right)\\
&\leq \frac1{\sqrt{2\pi}}\left(\sum_{n=1}^{N-1} \frac{r_n^{\kappa_a(n)}-r_n}{r_n-1}+\frac{1}{1-r_N}\sum_{n=N}^\infty \frac{2e}{n^2}\right)\\
&\leq \frac1{\sqrt{2\pi}}\left(\sum_{n=1}^{N-1} \frac{r_n^{\kappa_a(n)}-r_n}{r_n-1}+\frac{2e}{1-r_N}\zeta(2)\right)<\infty.
\end{align*}
Therefore, the series above converges. Thus, the infinite product converges absolutely at least for $|u|<\frac1a$ when $a\geq 2$.

Next, we show that the infinite product does not converge absolutely for $|u|\geq \frac1a$ when $a\geq 2$. It is easy to show that $\sum_{n=1}^\infty |(1-u^n)^{\kappa_a(n)}-1|$ diverges when $|u|\geq 1$. Hence, we may assume that $\frac1a\leq|u|<1$.

Firstly, we prove that it diverges when $u=\frac1a$.
\begin{align*}
\sum_{n=1}^\infty |(1-u^n)^{\kappa_a(n)}-1|
&=\sum_{n=1}^\infty \left(1-\left(1-\frac1{a^n}\right)^{\kappa_a(n)}\right)\\
&=\sum_{n=1}^\infty \frac{\kappa_a(n)}{a^n}\sum_{m=1}^{\kappa_a(n)} (-1)^{m+1}\frac1{\kappa_a(n)}\binom{\kappa_a(n)}{m}\frac1{a^{n(m-1)}}.
\end{align*}
Let $\displaystyle{I_n:=\sum_{m=1}^{\kappa_a(n)} (-1)^{m+1}\frac1{\kappa_a(n)}\binom{\kappa_a(n)}{m}\frac1{a^{n(m-1)}}}$ and let $n\geq 3$. Then, we have
\begin{align*}
    |I_n-1|&=\left|\sum_{m=2}^{\kappa_a(n)}(-1)^{m+1}\frac{\kappa_a(n)-1}{a^n}\cdots\frac{\kappa_a(n)-m+1}{a^n}\cdot\frac1{m!}\right|\\
    &\leq\sum_{m=2}^{\kappa_a(n)}\left|\frac{\kappa_a(n)-1}{a^n}\right|\cdots\left|\frac{\kappa_a(n)-m+1}{a^n}\right|\cdot\frac1{m!}.
\end{align*}
Here, for any $i\in\{1,\ldots,m-1\}$, we have
\[\left|\frac{\kappa_a(n)-i}{a^n}\right|\leq\frac1n+\frac1{a^{n-\left\lfloor n/2\right\rfloor-1}n}\leq\frac2n\]
by Proposition \ref{prop: order of kappa}. Hence,
\[|I_n-1| \leq \sum_{m=2}^{\kappa_a(n)}\frac{2^{m-1}}{n^{m-1}}\cdot\frac1{m!} \leq \sum_{m=2}^{\kappa_a(n)}\frac1{n^{m-1}} = \frac{1-n^{1-\kappa_a(n)}}{n-1}<\frac{1}{n-1}\leq\frac{1}{2}\]
for $n\geq 3$. Therefore, we have $I_n\geq\frac12$ for $n\geq 3$ and then
\begin{equation}\label{eq: evaluation u=1/a}
    \sum_{n=1}^\infty \left(1-\left(1-\frac1{a^n}\right)^{\kappa_a(n)}\right) \geq \sum_{n=3}^\infty \frac{\kappa_a(n)}{a^n} I_n \geq \sum_{n=3}^\infty\frac1{2n}\cdot\frac12 =\infty,
\end{equation}
since for $n\geq3$
\[\frac{\kappa_a(n)}{a^n}\geq \frac1n-\frac1{a^{n-\left\lfloor n/2\right\rfloor-1}n}\geq \frac1{2n}.\]

Secondly, we prove that $\sum_{n=1}^\infty |(1-u^n)^{\kappa_a(n)}-1|$ diverges for $\frac1a\leq |u|<1$.
Let $r=\frac{\Arg u}{2\pi}$ ($0\leq r<1$).
Assume that $r\in\mathbb{Q}$. We put $r=\frac{k_1}{k_2}$ ($k_1\in\mathbb{Z}$, $k_2\in\mathbb{N}$). Let $N=\{k_2 m\mid m\in\mathbb{N}\}$, then it holds that $u^n=|u|^n\geq \frac{1}{a^n}$ for any $n\in N$. Therefore, we have
\[\sum_{n=1}^\infty |(1-u^n)^{\kappa_a(n)}-1|\geq \sum_{n\in N} \left(1-\left(1-|u|^n\right)^{\kappa_a(n)}\right) \geq \sum_{n\in N} \left(1-\left(1-\frac1{a^n}\right)^{\kappa_a(n)}\right).\]
In the similar way as the inequality (\ref{eq: evaluation u=1/a}), we have
\[\sum_{n\in N}\left(1-\left(1-\frac1{a^n}\right)^{\kappa_a(n)}\right)\geq \frac14\sum_{n\in N\cap[3,\infty)}\frac1{n}.\]
Here, since the natural density of $N$ is $\frac{1}{k_2}>0$ and coincides with the Dirichlet density of $N$, the last infinite sum diverges. Thus, the series $\sum_{n=1}^\infty |(1-u^n)^{\kappa_a(n)}-1|$ diverges if $r\in\mathbb{Q}$.

Assume that $r\not\in\mathbb{Q}$. If $n\in\mathbb{N}$ satisfies $|1-u^n|\geq 1+\frac{1}{2a^n}$, it holds that
\begin{align*}
    \left|(1-u^n)^{\kappa_a(n)}-1\right|&\geq \left||1-u^n|^{\kappa_a(n)}-1\right| 
    \geq \left(1+\frac{1}{2a^n}\right)^{\kappa_a(n)}-1\\
    &>\frac{\kappa_a(n)}{2a^n} \overset{(\ast)}{\geq} \frac{1}{2n}\left(1-\frac{1}{a^{n-\left\lfloor n/2\right\rfloor-1}}\right) \geq \frac{1}{4n}
\end{align*}
by applying Proposition \ref{prop: order of kappa} to $(\ast)$. Hence, for every $N\subset\mathbb{N}$ which consists of $n\in\mathbb{N}$ satisfying $|1-u^n|\geq 1+\frac{1}{2a^n}$, we have
\[\sum_{n=1}^\infty |(1-u^n)^{\kappa_a(n)}-1| \geq \sum_{n\in N} |(1-u^n)^{\kappa_a(n)}-1| > \frac{1}{4}\sum_{n\in N} \frac{1}{n}.\]
Here, if the natural density of $N$ is positive, then the last infinite sum diverges and hence the series $\sum_{n=1}^\infty |(1-u^n)^{\kappa_a(n)}-1|$ diverges.
Therefore, it suffices to show that there exists $N\subset\mathbb{N}$ with positive natural density such that $|1-u^n|\geq 1+\frac{1}{2a^n}$ for any $n\in N$.

As described in Figure \ref{fig: theta}, let $p_n\in\mathbb{C}$ be the intersection of $|z|=|u|^n$ and $|z-1|=1+\frac{1}{2a^n}$ whose imaginary part is positive and let $\theta_n:=\Arg p_n$. Let $\Theta_n:=\{\theta\in\mathbb{R}\mid \theta_n\leq\theta\leq 2\pi-\theta_n\}$.
\begin{figure}[ht]
\centering
\begin{tikzpicture}[scale=2]
  \begin{scope}
    \fill[pattern=north west lines] (0,0) circle (1);
    \fill[black!7] (0,0) circle (1);
    \fill[white] (1,0) circle (1+1/2^3);
 \end{scope}
 \draw[->,>=stealth,very thick] (-1.2,0)--(2.4,0)node[above]{$\Re$}; 
 \draw[->,>=stealth,very thick] (0,-1.2)--(0,1.2)node[right]{$\Im$}; 
 \draw[name path=Cp,semithick,samples=100,domain=0:2*pi,variable=\t] plot({(1+1/2^3)*cos(\t r)+1},{(1+1/2^3)*sin(\t r)});
 \draw[semithick,samples=100,domain=0:2*pi,variable=\t] plot({cos(\t r)/2^2},{sin(\t r)/2^2});
 \draw[name path=C,semithick,samples=100,domain=0:2*pi,variable=\t] plot({cos(\t r)},{sin(\t r)});
 \draw[dotted, semithick,samples=100,domain=0:2*pi,variable=\t] plot({cos(\t r)+1},{sin(\t r)});
 \draw (1,0)node[above right]{$1$};
 \draw (2,0)node[below]{$2$};
 \draw (1/2^2,0)node[below=3.5mm, right=-2.2mm]{$\frac1{a^n}$};
 \draw ({2+1/2^3},0)node[below right]{$2+\frac1{2a^n}$};
 \draw[name path=Cu] (0,0) circle (6/7);
 \draw (6/7,0)node[below=3mm, left=-1mm]{$|u|^n$};
 \path[name intersections={of= Cu and Cp, by={pp,-pp}}];
 \fill[black] (pp) circle (0.025);
 \fill[black] (-1/2^3,0) circle (0.025);
 \draw[semithick] (0,0) to (pp);
 \coordinate[label=right:$p_n$] (pp1) at (0.7,0.85);
 \draw plot [smooth, tension=1] coordinates {(pp1) (0.5,0.86) (pp)};
 \draw[semithick] ([shift={(0,0)}]0:0.35) arc[radius=0.35, start angle=0, end angle= 74];
 \draw (-0.3,-0.5)node{$-\frac{1}{2a^n}$};
 \draw plot [smooth, tension=1] coordinates {(-1/2^3,0) (-0.2,-0.17) (-0.25,-0.36)};
 \draw (0.42,0.28)node{$\theta_n$};
\end{tikzpicture}
\caption{the definition of $p_n$ and $\theta_n$ in the complex plane}
\label{fig: theta}
\end{figure}
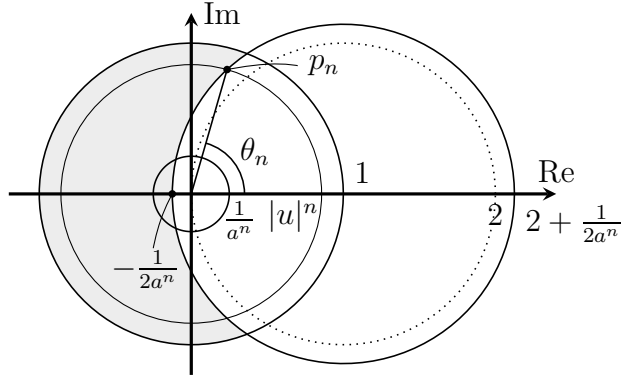
Note that if $\Arg(u^n)\in\Theta_n$, then such $n$ satisfies $|1-u^n|\geq 1+\frac{1}{2a^n}$. Since $\frac{1}{a}\leq |u|<1$, we have
\[\cos\theta_n =\frac{|u|^n}{2}-\frac{1}{2a^n |u|^n}-\frac{1}{8a^{2n}|u|^n} \geq \frac{1}{2a^n}-\frac{1}{2}-\frac{1}{8a^n} = \frac{3}{8a^n}-\frac{1}{2}\geq -\frac{1}{2}\]
for any $n\in\mathbb{N}$. Hence, we have $\theta_n\leq \frac{2}{3}\pi$ for any $n\in\mathbb{N}$. Therefore, it holds that $\Theta_n\supset \left[\frac{2}{3}\pi, \frac{4}{3}\pi\right]=:\Theta_\infty$ for any $n\in\mathbb{N}$. Let $N:=\{n\in\mathbb{N}\mid \Arg(u^n)\in \Theta_\infty\}$. Here, any $n\in N$ satisfies $|1-u^n|\geq 1+\frac{1}{2a^n}$, since $N\subset\{n\in\mathbb{N}\mid\Arg(u^n)\in\Theta_n\}$. Therefore, it suffices to show that the natural density of $N$ is positive.
We define $T_r\colon\mathbb{R}/\mathbb{Z}\to\mathbb{R}/\mathbb{Z}$ by $T_r(x)=x+r\bmod\mathbb{Z}$ and $\overline{\Theta}_\infty:=\left[\frac13,\frac23\right]\bmod\mathbb{Z}$.
Then, it holds that $N=\{n\in\mathbb{N}\mid T_r^n(0)\in\overline{\Theta}_\infty\}$.
Since $r\not\in\mathbb{Q}$, the continuous map $T_r$ on $\mathbb{R}/\mathbb{Z}$ is uniquely ergodic~\cite{EW2011ergodic-NT}*{Example 1.3 and Example 4.11}. Therefore, since $\overline{\Theta}_\infty$ is an interval in $\mathbb{R}/\mathbb{Z}$, we have
\[\lim_{M\to\infty}\frac{1}{M}\#\{n\in\mathbb{N}_0\cap[0,M)\mid T_r^n(x)\in\overline{\Theta}_\infty\}=m_{\mathbb{R}/\mathbb{Z}}(\overline{\Theta}_\infty)=\frac13\]
for every $x\in\mathbb{R}/\mathbb{Z}$ by \cite{EW2011ergodic-NT}*{Example 4.18 and Lemma 4.17} (also see \cite{EW2011ergodic-NT}*{Example 1.3}), where $m_{\mathbb{R}/\mathbb{Z}}$ is the Lebesgue measure on $\mathbb{R}/\mathbb{Z}$.
By putting $x=0$, we have
\[\lim_{M\to\infty}\frac{\#(N\cap[1,M])}{M} =\lim_{M\to\infty}\frac{1}{M}\#\{n\in\mathbb{N}_0\cap[0,M)\mid T_r^n(0)\in\overline{\Theta}_\infty\}=\frac13>0.\]
Therefore, the natural density of $N$ is positive.

Thus, the infinite product does not converge absolutely for $|u|\geq\frac1a$.
\end{proof}

\subsection{Proof of the main theorem}\label{subsection: proof of main thm}

Now, we prove Theorem \ref{result: AEP for Noetherian} using Lemma \ref{lem: core of AEP-formula}, Theorem \ref{thm: counting-dim}, and Lemma \ref{lem: core of AEP-convergence}.

\begin{proof}[Proof of Theorem \ref{result: AEP for Noetherian}]
At first, we derive the infinite product representation. By Theorem \ref{thm: property of AZF for Noether}, we have
\begin{align*}
\zeta_{\mcX/\fun}(s)&=\prod_{x\in\Ms{X}}\frac1{\left(1-\frac1s\right)^{\otimes r(x)}}=\prod_{x\in\Ms{X}}\prod_{j=0}^{r(x)}(s-r(x)+j)^{(-1)^{j+1}\binom{r(x)}{j}}\\
&= \prod_{x\in\Ms{X}}\prod_{j=0}^{r(x)} s^{(-1)^{j+1}\binom{r(x)}{j}} \left(1-\frac{r(x)-j}{s}\right)^{(-1)^{j+1}\binom{r(x)}{j}}\\
&= \left(\frac1s\right)^{\sum_{x\in\Ms{X}}\sum_{j=0}^{r(x)}(-1)^{j}\binom{r(x)}{j}} \prod_{x\in\Ms{X}}\prod_{j=0}^{r(x)} \left(1-\frac{r(x)-j}{s}\right)^{(-1)^{j+1}\binom{r(x)}{j}}.
\end{align*}
Since the counting function $N_\mcX(t)$ of $\mcX$ satisfies 
\[N_\mcX(t)=\sum_{x\in\Ms{X}}(t-1)^{r(x)}\in\mathbb{Z}[t]\]
by Remark \ref{rem: counting function of Noet F1}, we have
\[\sum_{x\in\Ms{X}}\sum_{j=0}^{r(x)}(-1)^{j}\binom{r(x)}{j}=\sum_{x\in\Ms{X}}(1-1)^{r(x)} =N_\mcX(1).\]
By Lemma \ref{lem: core of AEP-formula},
\[1-\frac{r(x)-j}{s}=\prod_{n=1}^\infty \left(1-\left(\frac1s\right)^n\right)^{\kappa_{r(x)-j}(n)}.\]
Hence,
\begin{align*}
\zeta_{\mcX/\fun}(s)
&= \left(\frac1s\right)^{N_\mcX(1)} \prod_{x\in\Ms{X}}\prod_{j=0}^{r(x)} \left(1-\frac{r(x)-j}{s}\right)^{(-1)^{j+1}\binom{r(x)}{j}}\\
&= \left(\frac1s\right)^{N_\mcX(1)} \prod_{x\in\Ms{X}}\prod_{j=0}^{r(x)} \left(\prod_{n=1}^\infty \left(1-\left(\frac1s\right)^n\right)^{\kappa_{r(x)-j}(n)}\right)^{(-1)^{j+1}\binom{r(x)}{j}}\\
&= \left(\frac1s\right)^{N_\mcX(1)} \prod_{n=1}^\infty \left(1-\left(\frac1s\right)^n\right)^{-\sum_{x\in\Ms{X}}\sum_{j=0}^{r(x)}(-1)^{j}\binom{r(x)}{j}\kappa_{r(x)-j}(n)}.
\end{align*}
We put
\[\kappa(n,\Ms{X}):=\sum_{x\in\Ms{X}}\sum_{j=0}^{r(x)}(-1)^j\binom{r(x)}{j}\kappa_{r(x)-j}(n)=\sum_{x\in\Ms{X}}\sum_{j=0}^{r(x)}(-1)^{r(x)-j}\binom{r(x)}{j}\kappa_j(n),\]
then we get the desired infinite product
\[\zeta_{\mcX/\fun}(s)=\left(\frac1s\right)^{N_\mcX(1)} \prod_{n=1}^\infty \left(1-\left(\frac1s\right)^n\right)^{-\kappa(n,\Ms{X})}.\]

Next, we show that the infinite product converges absolutely for $|s|>\reldim\Zs{X}/\mathbb{Z}$ if $\Zs{X}$ is of finite type over $\mathbb{Z}$.
In the calculation of the infinite product representation above, the point which is relevant to its convergence area is
\[1-\frac{r(x)-j}{s}=\prod_{n=1}^\infty \left(1-\left(\frac1s\right)^n\right)^{\kappa_{r(x)-j}(n)},\]
where we use Lemma \ref{lem: core of AEP-formula}. 
Let $|s|>\reldim\Zs{X}/\mathbb{Z}$. Since $\Zs{X}$ is of finite type over $\mathbb{Z}$, we have $\deg N_\mcX=\reldim\Zs{X}/\mathbb{Z}$ by Theorem \ref{thm: counting-dim}. Since
\[|s|>\reldim\Zs{X}/\mathbb{Z}=\deg N_\mcX=\max_{x\in\Ms{X}}r(x)\geq r(x)\geq r(x)-j\]
for any $x\in\Ms{X}$, we have $\frac1{|s|}<\frac1{r(x)-j}$ when $0\leq j<r(x)$. Therefore, when $j\ne r(x)$, 
\begin{equation}\label{eq: core-formula-conv}
\prod_{n=1}^\infty \left(1-\left(\frac1s\right)^n\right)^{\kappa_{r(x)-j}(n)}
\end{equation}
converges absolutely by Lemma \ref{lem: core of AEP-convergence}. Also, when $j=r(x)$, it converges absolutely since $\kappa_0(n)=0$. Thus, the infinite product
\begin{equation}\label{eq: AEP}
\prod_{n=1}^\infty \left(1-\left(\frac1s\right)^n\right)^{-\kappa(n,\Ms{X})}
\end{equation}
converges absolutely for $|s|>\reldim\Zs{X}/\mathbb{Z}$ if $\Zs{X}$ is of finite type over $\mathbb{Z}$.

Lastly, we show that the infinite product (\ref{eq: AEP}) diverges for $|s|\leq\reldim\Zs{X}/\mathbb{Z}$ if $\Zs{X}$ is of finite type over $\mathbb{Z}$. By Theorem \ref{thm: counting-dim}, we have $|s|\leq\max_{x\in\Ms{X}}r(x)$. When $x\in\Ms{X}$ and $j$ satisfy $r(x)-j=\max_{x\in\Ms{X}}r(x)$, the infinite product (\ref{eq: core-formula-conv}) diverges by Lemma \ref{lem: core of AEP-convergence}. Thus, the infinite product (\ref{eq: AEP}) diverges for $|s|\leq\reldim\Zs{X}/\mathbb{Z}$.
\end{proof}

%
%
%
%
%
\section{Application of the main theorem}\label{section: examples}
%
%
%
%
%

In this section, we apply our main theorem (Theorem \ref{result: AEP for Noetherian}) and Corollary \ref{cor: AEP-linear Mobius transform} to the cases of $\mathbb{A}^r$, $\mathbb{G}_m^r$ and toric varieties to obtain the absolute Euler products of the absolute zeta functions for the $\fun$-schemes in these cases. In fact, for the cases of $\mathbb{A}^r$ and $\mathbb{G}_m^r$, our result coincides with that of Kurokawa~\cite{kurokawa2016azf-eng}*{Exercise 7.2} which was calculated by a different method.

\subsection{Fundamental $\fun$-schemes}

\begin{eg}\label{eg: AEP for A^r}
Let $r\in\mathbb{N}$ and let $\fun[t_1,\ldots,t_r]:=\{0\}\cup\{t_1^{u_1}\cdots t_r^{u_r}\mid u_i\in\mathbb{N}_0\}$ be a monoid. Let $\textbf{A}^r:=\MSpec\fun[t_1,\ldots,t_r]$ be a monoid scheme.
Then, by the extension of the functors $\underline{\textbf{A}^r}$ and $\underline{\mathbb{A}^r}$, we obtain the functor $\mathcal{A}^r\colon\mr\to\Set$ satisfying that the geometric realisation of $\mathcal{A}^r|_{\Mo}$ (resp. $\mathcal{A}^r|_{\CR}$) is $\textbf{A}^r$ (resp. $\mathbb{A}^r$).
For the detail of the extension of functors, see \cite{CC2010}*{\S 4.2}. Moreover, $\mathcal{A}^r$ is a torsion free Noetherian $\fun$-scheme.

Since $\#\mathbb{A}^r(\mathbb{F}_{p^m})=p^{mr}$ for any $m\in\mathbb{N}$ and prime $p$, the counting function of $\mathcal{A}^r$ is $N_{\mathcal{A}^r}(t)=t^r\in\mathbb{Z}[t]$.
Hence, we have $\abschi(\mathbb{A}^r):=N_{\mathcal{A}^r}(1)=1$ and
\[\zeta_{\mathcal{A}^r/\fun}(s)=\abszeta{\mathbb{A}^r}(s)=\frac1{s-r}.\]
Since the prime ideals of $\fun[t_1,\ldots,t_r]$ are of the form
\[\mathfrak{p}_I=\bigcup_{i\in I} t_i\fun[t_1,\ldots,t_r],\]
where $I\subset\{1,\ldots,r\}$ and $\mathfrak{p}_\emptyset=(0)$, we have $r(\mathfrak{p}_I)=r-\#I$ for $\mathfrak{p}_I\in\textbf{A}^r$. We put
\begin{align*}
\kappa(n,\textbf{A}^r)&:=\sum_{\mathfrak{p}_I\in\textbf{A}^r}\sum_{j=0}^{r(\mathfrak{p}_I)}(-1)^{r(\mathfrak{p}_I)-j}\binom{r(\mathfrak{p}_I)}j\kappa_j(n)
=\sum_{i=0}^r\binom{r}{i}\sum_{j=0}^{r-i}(-1)^{r-i-j}\binom{r-i}j\kappa_j(n)\\
&=\sum_{j=0}^r \kappa_j(n)\sum_{i=0}^{r-j}(-1)^{r-j-i}\binom{r}{i}\binom{r-i}j
=\sum_{j=0}^r \kappa_j(n)\sum_{i=0}^{r-j}(-1)^{r-j-i}\binom{r}{j}\binom{r-j}i\\
&=\sum_{j=0}^r \binom{r}{j}\kappa_j(n)\sum_{i=0}^{r-j}(-1)^{r-j-i}\binom{r-j}i
=\sum_{j=0}^r \binom{r}{j}\kappa_j(n)(1-1)^{r-j}\\
&=\kappa_r(n).
\end{align*}
By Theorem \ref{result: AEP for Noetherian}, we obtain the absolute Euler product
\[\zeta_{\mathcal{A}^r/\fun}(s)=\frac1s\prod_{n=1}^\infty \left(1-\left(\frac1s\right)^n\right)^{-\kappa_r(n)}\]
and this infinite product converges absolutely for $|s|>\reldim\mathbb{A}^r/\mathbb{Z}=r$.
\end{eg}

\begin{eg}\label{eg: AEP for Gm}
Let $r\in\mathbb{N}$ and let $\fun[t_1^{\pm1},\ldots,t_r^{\pm1}]:=\{0\}\cup\{t_1^{u_1}\cdots t_r^{u_r}\mid u_i\in\mathbb{Z}\}$ be a monoid. Let $\textbf{G}_m^r:=\MSpec\fun[t_1^{\pm1},\ldots,t_r^{\pm1}]$ be a monoid scheme.
Then, by the extension of the functors $\underline{\textbf{G}_m^r}$ and $\underline{\mathbb{G}_m^r}$, we obtain the functor $\mathcal{G}_m^r\colon\mr\to\Set$ satisfying that the geometric realisation of $\mathcal{G}_m^r|_{\Mo}$ (resp. $\mathcal{G}_m^r|_{\CR}$) is $\textbf{G}_m^r$ (resp. $\mathbb{G}_m^r$). Moreover, $\mathcal{A}^r$ is a torsion free Noetherian $\fun$-scheme.

Since $\#\mathbb{G}_m^r(\mathbb{F}_{p^m})=(p^m-1)^r$ for any $m\in\mathbb{N}$ and prime $p$, the counting function of $\mathcal{G}_m^r$ is $N_{\mathcal{G}_m^r}(t)=(t-1)^r\in\mathbb{Z}[t]$.
Hence, we have $\abschi(\mathbb{G}_m^r):=N_{\mathcal{G}_m^r}(1)=0$ and
\[\zeta_{\mathcal{G}_m^r/\fun}(s)=\abszeta{\mathbb{G}_m^r}(s)=\prod_{k=0}^r (s-k)^{(-1)^{r-k+1}\binom rk}.\]
Since $\textbf{G}_m^r=\{(0)\}$ and $r((0))=r$, we put
\[\kappa(n,\textbf{G}_m^r):=\sum_{k=0}^r (-1)^{r-k}\binom rk\kappa_k(n).\]
By Theorem \ref{result: AEP for Noetherian}, we get the absolute Euler product
\[\zeta_{\mathcal{G}_m^r/\fun}(s)=\prod_{n=1}^\infty \left(1-\left(\frac1s\right)^n\right)^{-\sum\limits_{k=0}^r (-1)^{r-k}\binom rk\kappa_k(n)}\]
and this infinite product converges absolutely for $|s|>\reldim\mathbb{G}_m^r/\mathbb{Z}=r$.
\end{eg}

\subsection{Toric varieties}

Lastly, we calculate the absolute Euler products of the absolute zeta functions for the $\fun$-schemes associated with toric varieties, using their counting functions calculated by Deitmar~\cite{deitmar2008f1}.

At first, we review the notation of cones and fans~\cite{f1-land2009}.
Let $N$ be a \textit{lattice}, i.e. $N\cong\mathbb{Z}^r$, and let $N_\mathbb{R}:=N\otimes\mathbb{R}\cong\mathbb{R}^r$.
$\sigma\subset\mathbb{R}^r$ is called a \textit{polyhedral cone} if $v_1,\ldots,v_k\in N_\mathbb{R}$ such that $\sigma=v_1\mathbb{R}_{\geq 0}+\cdots+v_k\mathbb{R}_{\geq 0}$. In particular, if we can take linearly independent generators $v_1,\ldots,v_k$, $\sigma$ is called a \textit{simplicial cone}. A cone $\sigma$ is \textit{rational} if its generators lie in $N$. A cone $\sigma$ is \textit{strongly convex} if $\sigma\cap (-\sigma)=\{0\}$. We define the \textit{dimension} of $\sigma$ by $\dim\sigma:=\dim_\mathbb{R}(\sigma+(-\sigma))$. In this article, we call a strongly convex rational simplicial cone simply a \textit{cone}.

Let $N^\vee:=\Hom_\mathbb{Z}(N,\mathbb{Z})\cong\mathbb{Z}^r$ be its dual, and let $\langle\cdot,\cdot\rangle\colon N^\vee\times N\to\mathbb{Z}$ be the natural pairing. We define the dual of a cone $\sigma$ by $\sigma^\vee:=\{u\in N^\vee\otimes\mathbb{R}\mid \forall v\in\sigma,\ \langle u,v\rangle\geq 0\}$. $\tau$ is a \textit{face} of $\sigma$ if there exists $u\in\sigma^\vee$ satisfying $\tau=u^\perp\cap\sigma=\{w\in\sigma\mid\langle u,w\rangle=0\}$.
We denote a face $\tau$ of $\sigma$ by $\tau\prec\sigma$.
Let $\Phi$ be a nonempty set of cones in $N_\mathbb{R}$. $\Phi$ is a \textit{fan}, if every face of a cone in $\Phi$ is also in $\Phi$ and $\sigma_1\cap\sigma_2$ is a face of each $\sigma_1$ and $\sigma_2$ for any $\sigma_1, \sigma_2\in\Phi$.
Also, we call $\Phi$ to be \textit{finite} if $|\Phi|<\infty$.

Next, we review the definition of the $\fun$-scheme associated with a cone~\cite{f1-land2009}*{\S2.1}. Let $\sigma$ be a cone and $N$ be a lattice. Let $\sigma^\vee$ (resp. $N^\vee$) be the dual of $\sigma$ (resp. $N$) and $A_\sigma:=\sigma^\vee\cap N^\vee$ be a monoid. Then, we have the arithmetic scheme $\Spec\mathbb{Z}[A_\sigma]$ and the monoid scheme $\MSpec A_\sigma$. The $\fun$-scheme $\mcX_\sigma$ associated with the cone $\sigma$ is defined by the $\fun$-scheme satisfying that the geometric realisation of $\mcX_\sigma |_{\Mo}$ is $X_\sigma^{\Mo}:=\MSpec A_\sigma$ and that of $\mcX_\sigma|_\CR$ is $X_\sigma^\mathbb{Z}:=\Spec \mathbb{Z}[A_\sigma]$. Since $A_\sigma$ is finitely generated, $\mcX_\sigma$ is Noetherian.

Lastly, we review the definition of the $\fun$-scheme associated with a fan~\cite{f1-land2009}*{\S2.1}. Let $\Phi$ be a finite fan in $N_{\mathbb{R}}$. 
An inclusion $\tau\subset\sigma$ of cones induces an open immersions $\MSpec A_\tau\hookrightarrow\MSpec A_\sigma$ and $\Spec\mathbb{Z}[A_\tau]\hookrightarrow\Spec\mathbb{Z}[A_\sigma]$.
We define the monoid scheme and the arithmetic scheme associated with the fan $\Phi$ by
\[X_\Phi^{\Mo}:=\ilim{\sigma\in\Phi}\MSpec A_\sigma,\quad X_\Phi^{\mathbb{Z}}:=\ilim{\sigma\in\Phi}\Spec \mathbb{Z}[A_\sigma],\]
and we call $(X_\Phi^{\mathbb{Z}},\Phi)$ the \textit{toric variety} of the fan $\Phi$ of dimension $r$. The $\fun$-scheme $\mcX_\Phi$ associated with $\Phi$ is the $\fun$-scheme satisfying that the geometric realisation of $\mcX_\Phi |_{\Mo}$ is $X_\Phi^{\Mo}$ and that of $\mcX_\Phi |_\CR$ is $X_\Phi^{\mathbb{Z}}$. Since $\Phi$ is finite and $\MSpec A_\sigma$ is Noetherian, then $\mcX_\Phi$ is Noetherian.

Deitmar calculated the counting functions of the arithmetic schemes associated with cones and finite fans, which are equal to the counting functions of the $\fun$-scheme associated with them.

\begin{prop}[{\cite{deitmar2008f1}*{Proposition 4.3}}]\label{prop: counting function for cone/fan}
Let $N$ be a lattice of dimension $r$, and let $\Phi$ be a finite fan in $N_{\mathbb{R}}$. For any $\sigma\in\Phi$, let $\mcX_\sigma$ (resp. $\mcX_\Phi$) be the $\fun$-scheme associated with $\sigma$ (resp. $\Phi$). Then, the counting functions $N_{\mcX_\sigma}(t)$ and $N_{\mcX_\Phi}(t)$ are
\begin{align*}
N_{\mcX_\sigma}(t)&=\sum_{k=0}^{\dim\sigma} \#I_{\dim\sigma-k}^\sigma (t-1)^k = \sum_{j=0}^{\dim\sigma} \left(\sum_{k=j}^{\dim\sigma} (-1)^{k-j} \binom kj \#I_{\dim\sigma-k}^\sigma\right) t^j,\\
N_{\mcX_\Phi}(t)&=\sum_{k=0}^r \#I_{r-k}(t-1)^k = \sum_{j=0}^r \left(\sum_{k=j}^r (-1)^{k-j}\binom kj \#I_{r-k}\right) t^j,
\end{align*}
where $I_k^\sigma:=\{\eta\prec\sigma \mid \dim\eta=k\}$ and $I_k:=\{\sigma\in\Phi \mid \dim\sigma=k\}$.
\end{prop}

By this proposition, we have the absolute Euler product of the absolute zeta functions for the $\fun$-schemes associated with cones and finite fans.
\begin{eg}
Let $N$ be a lattice of dimension $r$, and let $\Phi$ be a finite fan in $N_{\mathbb{R}}$. For any $\sigma\in\Phi$, let $\mcX_\sigma$ (resp. $\mcX_\Phi$) be the $\fun$-scheme associated with $\sigma$ (resp. $\Phi$).
By Proposition \ref{prop: counting function for cone/fan}, the absolute zeta functions of $\mcX_\sigma$ and $\mcX_\Phi$ are
\begin{align*}
\zeta_{\mcX_\sigma/\fun}(s)&=\prod_{j=0}^{\dim\sigma}(s-j)^{\sum\limits_{k=j}^{\dim\sigma} (-1)^{k-j} \binom kj \#I_{\dim\sigma-k}^\sigma},\\
\zeta_{\mcX_\Phi/\fun}(s)&=\prod_{j=0}^r(s-j)^{\sum\limits_{k=j}^r (-1)^{k-j}\binom kj \#I_{r-k}},
\end{align*}
where $I_k^\sigma:=\{\eta\prec\sigma \mid \dim\eta=k\}$ and $I_k:=\{\sigma\in\Phi \mid \dim\sigma=k\}$. Using the homomorphism $M_n$ (Corollary \ref{cor: AEP-linear Mobius transform}), we can easily obtain the absolute Euler products of those absolute zeta functions. Since $N_{\mcX_\sigma}(1)=1$ and $N_{\mcX_\Phi}(1)=\#I_{r}$ by Proposition \ref{prop: counting function for cone/fan}, we have the absolute Euler products
\begin{align*}
\zeta_{\mcX_\sigma/\fun}(s)&=\frac1s\prod_{n=1}^\infty \left(1-\left(\frac1s\right)^n\right)^{-\sum\limits_{j=0}^{\dim\sigma} \left(\sum\limits_{k=j}^{\dim\sigma} (-1)^{k-j} \binom kj \#I_{\dim\sigma-k}^\sigma\right) \kappa_j(n)},\\
\zeta_{\mcX_\Phi/\fun}(s)&=\left(\frac1s\right)^{\#I_{r}}\prod_{n=1}^\infty \left(1-\left(\frac1s\right)^n\right)^{-\sum\limits_{j=0}^r \left(\sum\limits_{k=j}^r (-1)^{k-j}\binom kj \#I_{r-k}\right) \kappa_j(n)},
\end{align*}
and $\zeta_{\mcX_\sigma/\fun}(s)$ (resp. $\zeta_{\mcX_\Phi/\fun}(s)$) converges absolutely for $|s|>\dim\sigma$ (resp. $|s|>r$).
\end{eg}

\appendix

\subsection*{Acknowledgement}

The author thanks his advisor Kenichi Bannai for checking the draft of this article and giving many helpful comments. The author would like to thank Kazuki Yamada, Hohto Bekki, and Yoshinori Kanamura for reading it and giving many comments, and Tatsuya Oshita especially for providing the idea on another proof that $\kappa_a(n)\in\mathbb{Z}$ in Lemma \ref{lem: core of AEP-formula}. The author is also grateful to Yoshinosuke Hirakawa for pointing out errors in the arguments of the manuscript and giving many helpful comments and suggestions on the manuscript many times.
Lastly, the author thanks the referee for checking the article and giving valuable comments.

%
%
%
%
%
%
%
%
%
%

\bibliography{ref}

\end{document}